\documentclass[12pt]{amsart}
\usepackage[usenames,dvipsnames]{color}
\usepackage{amsmath,amsfonts,amssymb,amscd}
\calclayout
\makeatletter 

\@addtoreset{equation}{section}
\makeatother 
\newtheorem{theorem}{Theorem}[section]
\newtheorem{lemma}[theorem]{Lemma}

\newtheorem{proposition}[theorem]{Proposition}
\theoremstyle{remark}
\newtheorem*{remark}{Remark}

\newcommand{\grad}{\operatorname{grad}}
\newcommand\Alt{\operatorname{Alt}}
\renewcommand\L{\mathcal L}
\renewcommand\P{\mathcal P}
\newcommand\R{\mathbb{R}}
\newcommand\Th{\mathcal T_h}
\newcommand{\curl}{\operatorname{curl}}
\renewcommand{\div}{\operatorname{div}}
\newcommand{\tr}{\operatorname{tr}}
\newcommand{\rot}{\operatorname{rot}}
\newcommand{\dd}{\delta}
\def\<{\langle}
\def\>{\rangle}
\def\x{\times}
\def\0{\mathring}
\def\Bfrak{\mathfrak B}

\def\Hfrak{\mathfrak H}
\def\sigerr{\Sigma_h}
\def\uerr{U_h}
\def\uerrt{U_{h,t}}

\begin{document}
\title{Finite element exterior calculus for parabolic problems}
\author{Douglas N. Arnold}
\address{School of Mathematics, University of Minnesota, Minneapolis, MN 55455}
\email{arnold@umn.edu}
\author{Hongtao Chen}
\address{School of Mathematical Sciences, Xiamen University, Xiamen, 361005}
\email{chenht@xmu.edu.cn}
\thanks{The work of the first author was supported in part by NSF grant
DMS-1115291}
\begin{abstract}
In this paper, we consider the extension of the finite element exterior calculus from
elliptic problems, in which the Hodge Laplacian is an appropriate model problem,
to parabolic problems, for which we take the Hodge heat equation as our model problem.
The numerical method we study is a Galerkin method based on a mixed variational formulation
and using as subspaces the same spaces of finite element differential
forms which are used for elliptic problems.  We analyze both the semidiscrete
and a fully-discrete numerical scheme.
\end{abstract}
\maketitle

\section{Introduction}\label{sec:intro}
In this paper we consider the numerical solution of the \emph{Hodge heat equation}, the
parabolic equation associated to the Hodge Laplacian.  The initial-boundary value problem we study is
\begin{gather}
 u_t+(d\dd +\dd d)u=f \quad\text{in $\Omega\times (0,T]$},\label{hhe}\\
\tr(\star u)=0,\ \tr(\star du)=0 \quad\text{on $\partial\Omega\times(0,T]$},\label{bc}\\
u(\,\cdot\,,0)=u_0 \quad\text{in $\Omega$}.\label{ic}
\end{gather}
(We could consider other boundary conditions as well.)
Here the domain $\Omega\subset\mathbb{R}^n$ has a piecewise smooth, Lipschitz boundary,
the unknown $u$ is a time dependent differential $k$-form on $\Omega$, $u_t$ denotes its
partial derivative with respect to time, and $d$, $\delta$, $\star$, and $\tr$
denote the exterior derivative, coderivative, Hodge star, and trace
operators, respectively.

The numerical methods we consider are mixed finite element methods.  These
are based on the mixed weak formulation:
find $(\sigma,u):[0,T]\rightarrow
H\Lambda^{k-1} \times H\Lambda^k$, such that $u(0)=u_0$ and
\begin{gather}\label{mwf1}
 \<\sigma,\tau\>-\< d\tau,u\>=0, \quad \tau\in H\Lambda^{k-1},\ t\in(0,T],
\\\label{mwf2}
\< u_t,v\>+\< d\sigma,v\>+\< du, dv\>=\< f,v\>,  \quad v\in H\Lambda^k,\ t\in(0,T],
\end{gather}
(The notations are explained in the following section.)  Notice that, unlike in the
elliptic case, the harmonic forms need not be explicitly accounted for
in the weak formulation. The
well-posedness of the mixed formulation \eqref{mwf1},
\eqref{mwf2} is established in a precise sense in Theorem~\ref{wp}.

In the simplest case of $0$-forms ($k=0$), the differential equation
\eqref{hhe} is simply the heat equation,
$u_t-\Delta u = f$, and the boundary condition \eqref{bc}
is the Neumann boundary condition, $\partial u/\partial n=0$.
Moreover, in this case the space $H\Lambda^{k-1}$ vanishes,
and the weak formulation \eqref{mwf1}--\eqref{mwf2} is the usual (unmixed) one:
$u:[0,T]\to H^1(\Omega)$ satisfies
$$
\< u_t,v\> + \< \grad u,\grad v\> = \< f,v\>, \quad
v\in H^1(\Omega),\ t\in(0,T].
$$
In this case, the numerical methods and convergence results
obtained in this paper reduce to ones long known \cite{MR0277126,MR0351124}.
In the case of $n$-forms,
the differential equation is again the heat equation, although the natural boundary condition
is now the Dirichlet condition $u=0$.  In the case of $n$-forms,
the weak formulation seeks $\sigma\in H(\div)$, $u\in L^2$ such that
$$
\<\sigma,\tau\> - \<\div\tau, u\> = 0, \ \tau\in H(\div),
\quad
\<u_t,v\> - \<\div\sigma,v\> = \< f,v\>, \ v\in L^2, \quad t\in (0,T].
$$
This mixed method for the heat equation
was studied by Johnson and Thom\'ee in \cite{MR610597} in two dimensions.  Recently, Holst and Gillette
\cite{holst-gillette} have studied this mixed method in $n$-dimensions using a finite element exterior calculus
framework (in their work they consider hyperbolic problems as well).

For $k=1$ or $2$ in $n=3$ dimensions, the differential equation \eqref{hhe}
is the vectorial heat equation,
$$
u_t+\curl\curl u-\grad\div u=f.
$$
The weak formulations \eqref{mwf1}--\eqref{mwf2} for $k=1$ and $2$ correspond to two different
mixed formulations of this equation, the former using the scalar field $\sigma=\div u$ as
the second unknown, the latter using the vector field $\sigma=\curl u$.
For $k=1$, the boundary conditions \eqref{bc}, which are natural in the mixed
formulation, become
$u\cdot n=0$, $\curl u\x n=0$, while for $k=2$ these natural boundary conditions are
$u\x n=0$, $\div u=0$.  This vectorial heat equation arises, for example, in the
linearization of the Ginzburg--Landau equations for superconductivity \cite{GL},
and is related to the dynamical equations of Stokes and Navier--Stokes flow
(see, e.g., \cite{MR2784829}).

To discretize \eqref{mwf1}, \eqref{mwf2}, we utilize the two main families of finite element differential forms,
  the $\mathcal{P}_r\Lambda^k$ and
$\mathcal{P}_r^-\Lambda^k$ spaces.   Between them they include lots
of the best known families of finite elements on simplicial meshes
\cite[Section 5]{MR2594630}. We give both semidiscrete and fully
discrete schemes, and the corresponding convergence analysis.
Convergence rates under different norms  are shown in our final
results (see Theorem~\ref{semiconver} and Theorem~\ref{fullyconver} below). These achieve the
optimal rates allowed by the finite element spaces provided some
regularity assumptions are satisfied. These results also reveal the
relation between convergence rates under different norms and the
regularity of the exact solution.

The outline of the remainder of the paper is as follows. In Section~\ref{sec:prelim}, we review of basic notations
from finite element exterior calculus, including the two main families of finite element differential
forms, the $\P_r^-\Lambda^k$ and $\P_r\Lambda^k$ families, and some of their properties.
In Section~\ref{sec:ep}, we apply the elliptic theory to define an elliptic projection which will
be crucial to the error analysis of the time-dependent problem, and to obtain error estimates for it.
In Section~\ref{sec:wp}, we turn to the Hodge heat equation at the continuous level
and establish well-posedness of the mixed formulation.  We then give
a convergence analysis for the semidiscrete and fully discrete schemes in Sections~\ref{sec:sd} and \ref{sec:fd}, respectively.
Finally, we present some numerical examples confirming the results.

\section{Preliminaries}\label{sec:prelim}
We briefly review here some basic notions of
finite element exterior calculus for the Hodge Laplacian. Details can be found in
\cite[\S~2]{MR2269741} and
\cite[\S\S~3--4]{MR2594630} and in numerous references given there.

For $\Omega$ a domain in $\R^n$ and $k$ an integer,
let $L^2\Lambda^k=L^2\Lambda^k(\Omega)$ denote the Hilbert
space of differential $k$-forms on $\Omega$ with coefficients in $L^2$.  This is the space
of $L^2$ functions on $\Omega$ with values in $\Alt^k\R^n$, a finite dimensional Hilbert space
of dimension $\binom{n}{k}$
(understood to be $0$ if $k<0$ or $k>n$).
The \emph{Hodge star} operator $\star$
is an isometry of $\Alt^k\R^n$ and $\Alt^{n-k}\R^n$, and so induces an isometry
of $L^2\Lambda^k$ onto $L^2\Lambda^{n-k}$.  The inner product in $L^2\Lambda^k$
may be written $\< u,v\> = \int_\Omega u\wedge \star v$, with the corresponding
norm denoted $\|u\|$.
We view the exterior
derivative $d=d^k$ as a unbounded operator from $L^2\Lambda^k$ to $L^2\Lambda^{k+1}$.
Its domain, which we denote $H\Lambda^k(\Omega)$,
consists of forms $u\in L^2\Lambda^k$ for which the distributional
exterior derivative $du$ belongs to $L^2\Lambda^{k+1}$.
Assuming, as we shall, that $\Omega$
has Lipschitz boundary, the \emph{trace operator} $\tr=\tr_{\partial\Omega}$
maps $H\Lambda^k(\Omega)$ boundedly into an appropriate Sobolev space
on $\partial\Omega$ (namely $H^{-1/2}\Lambda^k(\partial\Omega)$).
The \emph{coderivative} $\delta$ is defined as $(-1)^{n(k+1)+1}\star d \star:H^*\Lambda^k
\to H^*\Lambda^{k-1}$, where $H^*\Lambda^k:= \star H\Lambda^{n-k}$.
The adjoint
$d^*=d^*_k$ of $d^{k-1}$ is the unbounded operator $L^2\Lambda^k\to L^2\Lambda^{k-1}$
given by restricting $\delta$ the domain of $d^*$,
$$
D(d^*) = \0H^*\Lambda^k:=\{\, u\in H^*\Lambda^k\,|\,\tr\star u =0\,\}.
$$

We denote by $\mathfrak{Z}^k$ and $\mathfrak{Z}^*_k$ the null spaces
of $d^k$ and $d^*_k$, respectively.  Their orthogonal complements
in $L^2\Lambda^k$ are $\mathfrak{B}^*_k$ and $\mathfrak{B}^k$,
the ranges of $d^*_{k+1}$ and $d^{k-1}$, respectively.  The orthogonal
complement of $\mathfrak{B}^k$ inside $\mathfrak{Z}^k$ is the
space of \emph{harmonic forms}
$$
\mathfrak{H}^k=\mathfrak{Z}^k\cap \mathfrak{Z}_k^{*}
= \{\, \omega\in H\Lambda^k(\Omega)\cap
\0H^*\Lambda^k(\Omega) \, |\, d\omega= 0, d^* \omega= 0\,\}.
$$
The dimension of $\mathfrak H^k$ is equal to the $k$th Betti number of $\Omega$,
so $\mathfrak H^k=0$ for $k\ne 0$ if $\Omega$ is contractible.
The \emph{Hodge
decomposition} of $L^2\Lambda^k$ and of $H\Lambda^k$ follow immediately:
\begin{gather}
L^2\Lambda^k=\mathfrak{B}^k\oplus\mathfrak{H}^k\oplus \mathfrak{B}_k^{*},\label{sec2:eq5}\\
H\Lambda^k=\mathfrak{B}^k\oplus \mathfrak{H}^k\oplus \mathfrak{Z}^{k\perp},\label{sec2:eq4}
\end{gather}
where $\mathfrak{Z}^{k\perp}=H\Lambda^k\cap\mathfrak{B}_k^{*}$ denotes the orthogonal complement of
$\mathfrak{Z}^k$ in $H\Lambda^k$.

The \emph{Hodge Laplacian} is the unbounded operator $L=dd^*+d^* d:
D(L)\subset L^2\Lambda^k\to L^2\Lambda^k$ with
\begin{equation*}
D(L)=\{\,v\in H\Lambda^k\cap \0H^*\Lambda^k\ \,|\,\ d^* v\in H\Lambda^{k-1}, dv\in
\0H^*\Lambda^{k+1}\,\}.
\end{equation*}
The null space of $L$ consists precisely of the harmonic forms $\mathfrak{H}^k$.

For any $f\in L^2\Lambda^k$, there exists a unique solution $u=Kf\in D(L)$
satisfying
$$
Lu=f \  (\operatorname{mod}\Hfrak), \quad u\perp \mathfrak{H}^k,
$$
(see \cite[Theorem 3.1]{MR2594630}). The solution $u$
satisfies the Hodge Laplacian boundary value problem
$$
(d\delta + \delta d ) u = f- P_{\mathfrak H}f \text{ in $\Omega$},\quad
\tr\star u =0, \ \tr \star du =0 \text{ on $\partial\Omega$},
$$
together with side condition $u\perp \mathfrak{H}^k$ required
for uniqueness.  The solution
operator $K$ is a compact
operator $L^2\Lambda^k\rightarrow H\Lambda^k\cap
\0H^*\Lambda^k$ and \emph{a fortiori}, is compact as an operator from
$L^2\Lambda^k$ to itself.

Now we consider the mixed finite element
discretization of the Hodge Laplacian boundary value problem,
following \cite{MR2594630}.
This is based on the mixed weak formulation, which
seeks $\sigma\in H\Lambda^{k-1}$, $u\in H\Lambda^k$, and
$p\in \mathfrak H^k$ such that
\begin{gather*}
\< \sigma,\tau\>-\< d\tau,u\>=0,\quad \tau\in  H\Lambda^{k-1},\\
\< d \sigma,v\>+\< du,dv\>+\< p,v\>=\< f,v\>,\quad  v\in H\Lambda^{k},\\
\< u,q\>=0,\quad q\in \mathfrak{H}^k.
\end{gather*}
It admits a unique solution given by $u=Kf$, $\sigma = d^*u$, $p=P_{\mathfrak H}f$.
We discretize the mixed formulation using Galerkin's
method.  For this, let $\Lambda^{k-1}_h$ and $\Lambda^k_h$
be finite dimension subspaces of $H\Lambda^{k-1}$ and $H\Lambda^k$, respectively,
satisfying $d \Lambda_h^{k-1}\subset \Lambda_h^k$.  We define the space of
\emph{discrete harmonic forms} $\mathfrak H^k_h$ as the orthogonal complement
of $\mathfrak B^k_h := d \Lambda_h^{k-1}$ inside
$\mathfrak Z^k_h:=\mathfrak Z\cap \Lambda_h^k$.  This immediately
gives the \emph{discrete Hodge decomposition}
\begin{equation*}
\Lambda_h^k=\mathfrak{B}_h^k\oplus \mathfrak{H}_h^k\oplus \mathfrak{Z}_h^{k\perp},
\end{equation*}
where $\mathfrak Z_h^{k\perp}$ is the orthogonal complement of $\mathfrak Z_h^k$
inside $\Lambda_h^k$.

The Galerkin method seeks
$\sigma_h\in \Lambda_h^{k-1}$, $u_h\in \Lambda_h^k$,
$p_h\in \mathfrak H_h^k$ such that
\begin{equation}\label{gal}
\begin{gathered}
\< \sigma_h,\tau\>-\< d\tau,u_h\>=0,\quad \tau\in  \Lambda^{k-1}_h,\\
\< d \sigma_h,v\>+\< du_h,dv\>+\< p_h,v\>=\< f,v\>,\quad  v\in \Lambda^{k}_h,\\
\< u_h,q\>=0,\quad q\in \mathfrak{H}^k_h.
\end{gathered}
\end{equation}

For the analysis of this discretization, we require the existence of a third space $\Lambda^{k+1}_h\subset H\Lambda^{k+1}$
which contains $d\Lambda^k_h$, so that $\Lambda^{k-1}_h\xrightarrow{d} \Lambda^k_h \xrightarrow{d}\Lambda^{k+1}_h$
is a subcomplex of the segment $H\Lambda^{k-1}\xrightarrow{d} H\Lambda^k \xrightarrow{d}H\Lambda^{k+1}$ of the de~Rham
complex.  Further we require that there exists a \emph{bounded cochain projection}, i.e., bounded linear
projection maps $\pi^j_h:H\Lambda^j\to \Lambda^j_h$, $j=k-1,k,k+1$, such that the diagram
\begin{equation}\label{bcp}
\begin{CD}
 H\Lambda^{k-1} @>d>> H\Lambda^k @>d>> H\Lambda^{k+1}
\\
@V\pi^{k-1}_hVV  @V\pi^{k}_hVV  @V\pi^{k+1}_hVV
\\
\Lambda^{k-1}_h @>d>> \Lambda^k_h @>d>>\Lambda^{k+1}_h
\end{CD}
\end{equation}
commutes.  A key result of the finite element exterior calculus is that,
under these assumptions, the Galerkin equations \eqref{gal} admit a unique solution
and provide a stable discretization.

Another important aspect of the finite element exterior calculus is the construction
of finite element spaces $\Lambda_h^k$ which satisfy these hypothesis, i.e.,
which combine to form de~Rham subcomplexes
with bounded cochain projections.  Let there be given a shape regular family of
meshes $\Th$ with mesh size $h$ tending to $0$.
For each $r\ge 1$, we define two finite element subspaces of
$H\Lambda^k$, denoted $\P_r\Lambda^k(\Th)$ and $\P_r^-\Lambda^k(\Th)$.
For $k=0$, these two spaces coincides and equal the degree $r$ Lagrange
finite element subspace of $H^1(\Omega)$.  For $k=n$, $\P_r^-\Lambda^n(\Th)$
coincides with $\P_{r-1}\Lambda^n(\Th)$, which may be viewed as the space
of all piecewise polynomials of degree at most $r-1$, without inter-element
continuity constraints.  However, for $0<k<n$,
$$
\P_{r-1}\Lambda^k(\Th) \subsetneq \P_r^-\Lambda^k(\Th) \subsetneq \P_r\Lambda^k(\Th).
$$
For stable mixed finite elements for the Hodge Laplacian, we have four
possibilities (which reduce to just one for $k=0$ and to two for $k=1$ or $n$):
\begin{equation}\label{spaces}
 \Lambda^{k-1}_h= \left\{
\begin{matrix}
 \P_{r}\Lambda^{k-1}(\Th)\\[1ex] \text{or} \\[1ex] \P_{r}^-\Lambda^{k-1}(\Th)
\end{matrix}
\right\},\quad
\Lambda^k_h=
 \left\{
\begin{matrix}
 \P_{r}^-\Lambda^k(\Th)\\[1ex] \text{or} \\[1ex] \P_{r-1}\Lambda^k(\Th) \text{ (if $r>1$)}
\end{matrix}
\right\}.
\end{equation}
As the auxiliary space, if $\Lambda^k_h = \P_r^-\Lambda^k(\Th)$, we take $\Lambda^{k+1}_h=\P_r^-\Lambda^{k+1}(\Th)$,
while if $\Lambda^k_h = \P_{r-1}\Lambda^k(\Th)$, we take $\Lambda^{k+1}_h=\P_{r-1}^-\Lambda^{k+1}(\Th)$.

For this choice of spaces, it is known
(\cite[\S~5.4]{MR2269741}, \cite[\S~5.5]{MR2594630}, \cite{MR2373181}) that there
exist cochain projections as in \eqref{bcp} for which $\pi^j_h:L^2\Lambda^j\to\Lambda_h^j$
is bounded in $L^2\Lambda^j$ uniformly with respect to $h$.  In particular, this implies
that there is a constant $C$ independent of $h$ such that
\begin{equation}\label{bestapprox}
  \|u-\pi^j_h u\|\le C\inf_{v\in\Lambda^j_h}\|u-v\|,\quad u\in L^2\Lambda^j.
\end{equation}
Moreover, we have the approximation estimates
\begin{equation}\label{l2approx}
\|u-\pi^j_hu\| \le C h^s \|u\|_s, \quad 0\le s \le
\begin{cases}
r, & \Lambda^j_h=\P_r^-\Lambda^j(\Th),\\
r+1, & \Lambda^j_h=\P_r\Lambda^j(\Th).
\end{cases}
\end{equation}
Note that we use $\|u\|_s$ as a notation for the Sobolev norm $\|u\|_{H^s\Lambda^j}$.

\section{Elliptic projection of the exact solution}\label{sec:ep}
As usual, we shall obtain error estimates for the finite element approximation to the
evolution equation by comparing it to an appropriate elliptic projection of the exact
solution into the finite element space.  In this section we define the elliptic projection
and establish error estimates for it.

Given any $u\in D(L)$, the elliptic
projection of $u$ is defined as
$(\hat{\sigma}_h,\hat{u}_h,\hat{p}_h)\in \Lambda_h^{k-1}\times
\Lambda_h^k\times \mathfrak{H}_h^k$, such that
\begin{gather}
\< \hat{\sigma}_h,\tau\>-\< d\tau,\hat{u}_h\>=0,\quad \tau\in  \Lambda_h^{k-1},\label{ep1}\\
\< d \hat{\sigma}_h,v\>+\< d\hat{u}_h,dv\>+\< \hat{p}_h,v\>=\< Lu,v\>,\quad  v\in \Lambda_h^{k},\label{ep2}\\
\< \hat{u}_h,q\>=\< u,q\>,\quad q\in \mathfrak{H}_h^k.\label{ep3}
\end{gather}
By Theorem 3.8 of \cite{MR2594630} there exists a unique solution to
\eqref{ep1}--\eqref{ep3}.  Now we follow the approach of \cite{MR2594630} to
derive error estimates.  To this end, we introduce some notations.  First, let
$P_{\mathfrak{H}_h}: L^2\Lambda^k\rightarrow \mathfrak{H}_h^k$ denote the $L^2$-projection.
From \eqref{ep3}, $P_{\mathfrak{H}_h}\hat{u}_h=P_{\mathfrak{H}_h}u$.
Moreover, from \cite[Section 3.4]{MR2594630},
$$
\hat{p}_h=P_{\mathfrak{H}_h}(L u)= P_{\mathfrak{H}_h}(d\sigma),
$$
where $\sigma=d^*u$, the last equality holding because $d^* du\in
\mathfrak{B}^*_k\perp \mathfrak{Z}^k$, but $\mathfrak{H}^k_h\subset \mathfrak{Z}_h^k\subset
\mathfrak{Z}^k$.

Next, define $\beta=\beta_h^k$, $\mu=\mu_h^k$, and $\eta=\eta_h^k$ by
\begin{gather*}
\beta=\|(I-\pi_h)K\|_{\L(L^2\Lambda^k,L^2\Lambda^k)},\
\mu=\|(I-\pi_h)P_{\mathfrak{H}^k}\|_{\L(L^2\Lambda^k,L^2\Lambda^k)},\\
 \eta=\max_{j=0,1}\max[\|(I-\pi_h)dK\|_{\L(L^2\Lambda^{k-j},L^2\Lambda^{k-j+1})},\|(I-\pi_h)d^*
 K\|_{\L(L^2\Lambda^{k+j},L^2\Lambda^{k+j-1})}].
\end{gather*}
From \eqref{bestapprox} and the compactness of $K:L^2\Lambda^k\to H\Lambda^k\cap \0H^*\Lambda^k$,
we conclude that $\eta,\beta,\mu\rightarrow 0$ as $h\rightarrow
0$.  Assuming $H^2$ regularity for the Hodge Laplacian (by which we mean
both that $\|Kf\|_2\le C\|f\|_0$ for all $f\in L^2\Lambda^k$ and that
$\mathfrak{H}^k\subset H^2$), then we have
\begin{equation}\label{etabeta}
\eta = O(h), \quad \beta,\mu=O(h^{\min(2,r)})
\end{equation}
for any of the choices of spaces in \eqref{spaces}.  Note that $\beta=O(h^2)$ except
in the case $r=1$ and so $\Lambda^k_h=\P_1^-\Lambda^k$.

Finally, we denote the best approximation error in the $L^2$ norm by
\begin{equation*}
 E(w)=\inf_{v\in \Lambda_h^k}\|w-v\|,\quad w\in L^2\Lambda^k,\ k=0,\dots,n.
\end{equation*}
We are now ready to give the error estimates for the elliptic projection.

\begin{theorem}\label{ellp}
 Let $u\in D(L)$ and let
$(\hat{\sigma}_h,\hat{u}_h)$ be defined by \eqref{ep1}--\eqref{ep3}. Then we have
\begin{gather}
 \|d(\sigma-\hat{\sigma}_h)\|\leq CE(d\sigma),\label{epest1}\\
\|\sigma-\hat{\sigma}_h\|\leq C(E(\sigma)+\eta E(d\sigma)),\label{epest2}\\
\|\hat{p}_h\|\leq C\mu E(d\sigma),\label{epest3}\\
\|d(u-\hat{u}_h)\|\leq C(E(du)+\eta E(d\sigma)),\label{epest4}\\
\|u-\hat{u}_h\|\leq C(E(u)+E(P_{\mathfrak{H}}u)+\eta [E(du)+E(\sigma)]
+(\eta^2+\beta)E(d\sigma)+\mu E(P_{\mathfrak{B}}u)).\label{epest5}
\end{gather}
\end{theorem}
\begin{proof}
This is essentially proven in \cite{MR2594630}, except that there it is assumed that
$u\perp\mathfrak{H}$ and $\hat u_h\perp\mathfrak{H}_h$.  To account for this difference,
let $\tilde{u}=u-P_{\mathfrak{H}}u$ and $\tilde{u}_h=\hat{u}_h-P_{\mathfrak{H}_h}\hat{u}_h$.
Then \eqref{ep1} and
\eqref{ep2} continue to hold with $u$ and $u_h$ replaced by $\tilde u$ and $\tilde u_h$, respectively,
and, in place of \eqref{ep3}, we have
$$
\<\tilde{u}_h, q\>=0, \quad q\in \Hfrak_h^k.
$$
Application of Theorem~3.11 of \cite{MR2594630} (with $f=Lu$ and $p=0$) then gives
the (\ref{epest1}--\ref{epest4}), and, instead of \eqref{epest5}, we get
\begin{align*}
 \|\tilde{u}-\tilde{u}_h\|&\leq C(E(\tilde{u})+\eta
[E(du)+E(\sigma)]+(\eta^2+\beta)E(d\sigma)+\mu
E(P_{\mathfrak{B}}u))\\
&\leq C(E(u)+E(P_{\Hfrak}u)+\eta
[E(du)+E(\sigma)]+(\eta^2+\beta)E(d\sigma)+\mu
E(P_{\mathfrak{B}}u)).
\end{align*}
Thus $\|\tilde{u}-\tilde{u}_h\|$ is bounded by the right-hand side of \eqref{epest5},
and, to complete the proof, it suffices
bound $P_{\mathfrak{H}}u-P_{\mathfrak{H}_h}\hat{u}_h$ by same quantity.
Now
$$
P_{\mathfrak{H}}u-P_{\mathfrak{H}_h}\hat{u}_h=P_{\mathfrak{H}}u-P_{\mathfrak{H}_h}u=
(I-P_{\mathfrak{H}_h})P_{\mathfrak{H}}u - P_{\mathfrak{H}_h}(u-P_{\mathfrak{H}}u),
$$
For the first term on the right-hand side, we use \cite[Theorem 3.5]{MR2594630} and \eqref{bestapprox}
to get
$$
\|(I-P_{\mathfrak{H}_h})P_{\mathfrak{H}}u\|\leq
\|(I-\pi_h)P_{\mathfrak{H}}u\|\leq C E(P_{\mathfrak{H}}u).
$$
To estimate the second term, we use the Hodge
decomposition \eqref{sec2:eq4} to write $u-P_{\mathfrak{H}}u=u_b+u_{\perp}$
with $u_b\in \mathfrak{B}^k, u_{\perp}\in \mathfrak{Z}^{k\perp}$.
Since $\mathfrak{H}_h^k\subset \mathfrak{Z}^k$,
$P_{\mathfrak{H}_h}u_{\perp}=0$, and since $\pi_hu_b\in
\mathfrak{B}_h^k$, $P_{\mathfrak{H}_h}\pi_hu_b=0$.  Hence
$P_{\mathfrak{H}_h}(u-P_{\mathfrak{H}}u)=P_{\mathfrak{H}_h}(I-\pi_h)u_b$.
We normalize this quantity by setting
$$
q=P_{\mathfrak{H}_h}(u-P_{\mathfrak{H}}u)/\|P_{\mathfrak{H}_h}(u-P_{\mathfrak{H}}u)\|\in \mathfrak{H}_h^k.
$$
Then $P_{\mathfrak H} q\in\mathfrak H$, and, by
\cite[Theorem 3.5]{MR2594630},
$\|q-P_{\mathfrak H} q\|\leq \|(I-\pi_h)P_{\mathfrak H} q\|\le\mu$.
Therefore,
\begin{equation*}
 \|P_{\mathfrak{H}_h}(u-P_{\mathfrak{H}}u)\|=(P_{\mathfrak{H}_h}(u-P_{\mathfrak{H}}u),q)
=(P_{\mathfrak{H}_h}(I-\pi_h)u_b,q)=((I-\pi_h)u_b,q).
\end{equation*}
Now $(I-\pi_h)u_b\in\mathfrak B^k$, and so is orthogonal to $\mathfrak H$.  Thus
$$
 ((I-\pi_h)u_b,q)= ((I-\pi_h)u_b,q-P_{\mathfrak H} q)\leq \|(I-\pi_h)u_b\|\|q-P_{\mathfrak H} q\|
\leq C\mu E(P_{\mathfrak{B}}u),
$$
by  \eqref{bestapprox}.
Combining these results, we get
$$
\|P_{\mathfrak{H}}u-P_{\mathfrak{H}_h}\hat{u}_h\|\leq C[E(P_{\mathfrak{H}}u)+\mu E(P_{\mathfrak{B}}u)],
$$
completing the proof of the theorem.
\end{proof}

Assuming sufficient regularity of $u$ and $\sigma=d^*u$, we can combine the estimates of the theorem
with the approximation results of \eqref{l2approx} to obtain rates of convergence for the elliptic projection.
The precise powers of $h$ and Sobolev norms that arise depend on the particular choice of spaces
in \eqref{spaces}.  For example, if we take
$\Lambda^{k-1}_h=\P_r\Lambda^{k-1}(\Th)$, then we can show the optimal
estimate $\|\sigma-\sigma_h\|\le C h^{r+1}\|\sigma\|_{r+1}$,
but,  if $\Lambda^{k-1}_h=\P_r^-\Lambda^{k-1}(\Th)$, then clearly we can only have
$\|\sigma-\sigma_h\|=O(h^r)$.  Rather than give a complicated
statement of the results, covering all the possible cases,
in the following theorem and below we restrict to a particular choice of spaces
from among the possibilities in \eqref{spaces}.
Moreover, we assume $r>1$, since the case $r=1$ is slightly different.  However, very similar results
can be obtained for any of the choices of spaces permitted in \eqref{spaces}, including for $r=1$,
in the same way.  Finally, we introduce the space
$$
\bar H^r = \{\, u\in H^r\,|\, P_{\mathfrak H} u\in H^r, \ P_{\mathfrak B} u\in H^{r-2}\,\},
$$
with the associated norm
$$
\|u\|_{\bar H^r}= \|u\|_r + \|P_{\mathfrak H} u\|_r+\|P_{\mathfrak B} u\|_{r-2},
$$
since it will arise frequently below.
\begin{theorem}\label{eprate}
Assume $H^2$ regularity for the Hodge Laplacian, so \eqref{etabeta} holds and suppose that we
use the finite element spaces $\Lambda^{k-1}_h=\P_r^-\Lambda^{k-1}(\Th)$ and $\Lambda^k_h=\P_r^-\Lambda^k(\Th)$
(so that the auxilliary space is $\Lambda^{k+1}_h=\P_r^-\Lambda^{k+1}(\Th)$), for some $r>1$.
Then we have the following convergence rates for the elliptic projection:
\begin{gather*}
\|d(\sigma-\hat{\sigma}_h)\|\leq Ch^r\|d\sigma\|_r,\\
\|\sigma-\hat{\sigma}_h\|\leq
 C h^r\|\sigma\|_r,\\
\|\hat{p}_h\|\leq Ch^r\|d\sigma\|_{r-2},\\
\|d(u-\hat{u}_h)\|\leq C h^{r}(\|du\|_{r} + \|d\sigma\|_{r-1}),\\
\|u-\hat{u}_h\|\leq C h^r\|u\|_{\bar H^r}.
\end{gather*}
\end{theorem}

\section{Well-posedness of the parabolic problem}\label{sec:wp}
We now turn to the Hodge heat equation.
In this section we demonstrate well-posedness of the initial-boundary
value problem \eqref{mwf1}, \eqref{mwf2}.  The
key tool is the
Hille--Yosida--Phillips theory as presented, for example, in\cite{MR2759829}
and \cite{MR1691574}.

We begin by showing that the Hodge Laplacian is maximal
monotone, or, equivalently, in the terminology of \cite{MR1691574}, that its negative is
m-dissipative).  This is the key hypothesis needed to apply the Hille--Yosida--Phillips theory to
the problem \eqref{hhe}--\eqref{ic}.
\begin{theorem}
 The Hodge Laplacian $L$ is maximal monotone.  That is, it satisfies
$$
 \< Lv,v\>\geq 0, \quad \forall v\in D(L),
$$
and, for any $f\in L^2\Lambda^k$, there exists $u\in D(L)$ such that
$u+Lu=f$.
\end{theorem}
\begin{proof}
For any $v\in D(L)$, $\< Lv,v\>=\< dv,dv\>+\< d^*v,d^*v\>$, so
the monotonicity inequality is obvious.  Now, for any $f\in L^2\Lambda^k$,
the Riesz representation theorem furnishes a unique
$u\in H\Lambda^k\cap \0H^*\Lambda^k$
such that
\begin{equation}\label{wf}
\< du,dv\>+\<d^*u,d^*v\>+\<u,v\>=\<f,v\>, \quad v\in H\Lambda^k\cap \0H^*\Lambda^k.
\end{equation}
We shall show that this $u$ belongs to $D(L)$, from which it follows immediately
that $u+Lu=f$.

To show that $u\in D(L)$, we must show that $du\in \0H^*\Lambda^{k+1}$ and
$d^*u\in H\Lambda^{k-1}$.  From \eqref{wf}, $f-u$ is orthogonal to $\Hfrak^k$,
so, using the Hodge decomposition of $L^2\Lambda^k$, we may write
$f-u=df_1+d^*f_2$ with $f_1\in H\Lambda^{k-1}\cap \Bfrak^*_{k-1}$ and
$f_2\in \0H^*\Lambda^{k+1}\cap \Bfrak^{k+1}$.
Then
$$
\<f-u,v\>=\<df_1 + d^*f_2,v\>=\<f_1,d^*v\>+\<f_2,dv\>, \quad v\in H\Lambda^k\cap \0H^*\Lambda^k.
$$
Combining with \eqref{wf}, we get
\begin{equation}\label{t}
\<du-f_2,dv\>+\<d^*u-f_1,d^*v\> =0, \quad v\in H\Lambda^k\cap \0H^*\Lambda^k.
\end{equation}
Now $du,f_2\in \Bfrak^{k+1}$, so there exists $v\in \mathfrak{Z}^{k\perp}=H\Lambda^k\cap \mathfrak{B}^*_k$ such that
$dv=du-f_2$.  Choosing this $v$ in \eqref{t}, we find
$du=f_2\in \0H^*\Lambda^{k+1}$, as desired.  Similarly $d^* u=f_1\in H\Lambda^{k-1}$.
\end{proof}

Since $L$ is maximal monotone and self-adjoint, we obtain the following existence theorem.
This is proved in \cite{MR1691574} in Theorems~3.1.1 and 3.2.1 for $f=0$ and $u_0\in L^2\Lambda^k$, and in
Proposition~4.1.6 for general $f$ and $u_0$ in $D(L)$.  Combining the two results by superposition, gives the
theorem.
\begin{theorem}\label{sf}
 Suppose that $u_0\in L^2\Lambda^k$ and $f\in C([0,T];L^2\Lambda^k)$ are given and that either
$f\in L^1((0,T); D(L))$ or $f\in W^{1,1}((0,T);L^2\Lambda^k)$.  Then there exists a unique
$u\in C([0,T];L^2\Lambda^k)\cap C((0,T];D(L))\cap C^1((0,T];L^2\Lambda^k)$, such that
$$
u_t + Lu = f \text { on $\Omega\x (0,T]$}, \quad u(0)=u_0.
$$
If further, $u_0\in D(L)$, then $u\in C([0,T];D(L))\cap C^1([0,T];L^2\Lambda^k)$.
\end{theorem}
We denote by $S(t):L^2\Lambda^k\to L^2\Lambda^k$ the solution
operator for the homogeneous problem ($f\equiv0$),
so $u(t)=S(t)u_0$ solves $u_t+Lu=0$, $u(0)=u_0$.  Then $S(t)$ is
a contraction in $L^2\Lambda^k$ for all $t\in[0,T]$, i.e.,
$\|S(t)\|\le1$,
and $S(t)$ commutes with $L$ on $D(L)$ (Theorem~3.1.1 of \cite{MR1691574}).

We can measure the regularity of the solution (for general $f$)
by using the iterated domains defined
by $D(L^l) = \{\,u\in D(L^{l-1})\,|\,L^{l-1}u\in D(L)\,\}$, $l\ge 2$.
The next theorem show that if $f$ is more regular, then the solution is also more
regular.
\begin{theorem}\label{sreg}
Suppose that in addition to the hypotheses of Theorem~\ref{sf}, we
have that $f$ belongs to $C((0,T];D(L))\cap L^1((0,T);D(L^2))$.  Then
\begin{equation}\label{reg}
u\in C^1((0,T];D(L)).
\end{equation}
\end{theorem}
\begin{proof}
If $f=0$, then \cite[Theorem~7.7]{MR2759829} implies that
\begin{equation*}
u\in C^k((0,T];D(L^l)),
\end{equation*}
for all $k,l\ge 0$.
Therefore, it is sufficient to treat the case $u_0=0$, which we do using
Duhamel's principle.  By Proposition~4.1.6 of \cite{MR1691574}, the solution
is given by
\begin{equation*}
u(t)=\int_0^tS(t-s)f(s)ds
\end{equation*}
in this case, and, assuming that $f$ satisfies the hypotheses of
Theorem~\ref{sf},
$$
u\in C([0,T];D(L))\cap C^1([0,T];L^2\Lambda^k).
$$
Now $f\in L^1((0,T);D(L^2))$, so
$$
L^2u(t) = \int_0^t S(t-s)L^2 f(s)\,ds,
$$
by the commutativity of $S(t-s)$ and $L$.
Since $S(t-s)$ is a contraction in $L^2\Lambda^k$, this implies that $u\in C([0,T];D(L^2))$
and so $Lu\in C([0,T];D(L))$.
Since we also assume that $f\in C((0,T];D(L))$, \eqref{reg} follows immediately
from the equation $u_t = f-Lu$.
\end{proof}

Next we show that the
solution $u$ guaranteed by Theorem \ref{sf}, together with
$\sigma=d^*u$, is a solution of the mixed problem
\eqref{mwf1}, \eqref{mwf2}. Since $u\in C((0,T];D(L))$,
$\sigma=d^*u\in C((0,T];H\Lambda^{k-1})$ and \eqref{mwf1}
holds.  Clearly
$$
\<u_t,v\> + \<Lu,v\> = \<f,v\>, \quad v\in L^2\Lambda^k, \ t\in (0,T].
$$
Since $u\in C((0,T];D(L))$, we have
$$
\<Lu,v\> = \<dd^* u,v\> + \<d^*du,v\> = \<d\sigma ,v\> + \<du,dv\>, \quad v\in H\Lambda^k, \ t\in(0,T].
$$
Combining the last two equations gives \eqref{mwf2}.

We are now ready to state the main result for this section.
\begin{theorem}\label{wp}
 Suppose that $u_0\in L^2\Lambda^k$ and $f\in C([0,T];L^2\Lambda^k)$ are given and that either
$f\in L^1((0,T); D(L))$ or $f\in W^{1,1}((0,T);L^2\Lambda^k)$.  Then there exist unique
$$
\sigma\in C((0,T];H\Lambda^{k-1}), \quad u\in  C([0,T];L^2\Lambda^k)\cap C((0,T];D(L))\cap C^1((0,T];L^2\Lambda^k),
$$
satisfying the mixed problem \eqref{mwf1}, \eqref{mwf2} and the initial condition $u(0)=u_0$.
If, moreover, the hypotheses of Theorem~\ref{sreg} are satisfied, then \eqref{reg} holds.
\end{theorem}
\begin{proof}
 We have already established existence. For uniqueness, we assume $f=0$ and
take $\tau=\sigma$ in \eqref{mwf1}
and $v=u$ in \eqref{mwf2}, to obtain
$$
\frac 12 \frac d{dt}\|u\|^2 = -\|\sigma\|^2 -\|du\|^2 \le 0.
$$
Therefore $\|u\|^2$ is decreasing in time, so if $u(0)=0$, then $u\equiv 0$.
Finally, \eqref{mwf1} then implies that $\sigma\equiv0$.
\end{proof}

\section{The semidiscrete finite element method}\label{sec:sd}
The semidiscrete finite element method for the Hodge heat equation is Galerkin's method applied to the
mixed variational formulation \eqref{mwf1}, \eqref{mwf2}.  That is, we choose finite
element spaces $\Lambda_h^{k-1}$
and $\Lambda_h^k$ as in \eqref{spaces} for some value of $r\ge 1$, and seek $(\sigma_h,u_h)\in
C([0,T];\Lambda_h^{k-1})\times C^1([0,T];\Lambda_h^k)$, such that $u_h(0)=u_h^0$, a given
initial value in $\Lambda^k_h$, and
\begin{gather}
\<\sigma_h,\tau\>-\< d\tau,u_h\>=0,\quad \tau\in \Lambda_h^{k-1},\ t\in(0,T],\label{sd1}\\
\< u_{h,t},v\>+\< d\sigma_h, v\>+\< du_h, dv\>=\< f,v\>, \quad  v\in \Lambda_h^k,\ t\in(0,T].\label{sd2}
\end{gather}
In this section we shall establish convergence estimates for this scheme.

We may interpret the semidiscrete solution in terms of two operators,
$d^*_h: \Lambda_h^k\rightarrow \Lambda_h^{k-1}$ and
$L_h: \Lambda_h^k\rightarrow \Lambda^k_h$, which are discrete analogues
of $d^*$ and $L$, respectively.  For $v\in\Lambda^k_h$,  $d^*_hv\in\Lambda_h^{k-1}$ is
defined by the equation
\begin{equation*}
 \<d^*_hv,\tau\>=\< v,d\tau\>, \quad \tau\in \Lambda_h^{k-1},
\end{equation*}
and the discrete Hodge Laplacian
$L_h: \Lambda_h^k\rightarrow \Lambda^k_h$ is given by
$L_h=d^*_hd+dd^*_h$.  The following characterization
is then a direct consequence of the definitions.
\begin{lemma} The pair $(\sigma_h,u_h)\in
C([0,T];\Lambda_h^{k-1})\times C^1([0,T];\Lambda_h^k)$ solves
\eqref{sd1} and \eqref{sd2} if and only if $u_h(t)\in
C^1([0,T];\Lambda_h^k)$ solves
\begin{equation}\label{ODE}
 u_{h,t}+L_hu_h=P_hf,\quad 0\leq t\leq T,
\end{equation}
where $P_h$ is $L^2$ projection of $f$ onto $\Lambda_h^k$, and $\sigma_h=d_h^*u_h$.
\end{lemma}
%
From the theory of ordinary differential equations, there exists
a unique solution $u_h\in C^1([0,T];\Lambda_h^k)$ solving the ODE \eqref{ODE} and taking
a given initial value. Letting $\sigma_h=d^*_hu_h$, we obtain a unique solution to
the semidiscrete finite element scheme \eqref{sd1}, \eqref{sd2}.

\begin{remark}
 The formulation \eqref{ODE} is useful for theoretical purposes, but is
typically not implemented directly, rather
only implicitly via the mixed method.  This is because the operator $d^*_h$ is not local.
Even if the finite element function $v$ is supported in just a few elements, $d^*_hv$ will
generally have global support.
\end{remark}

Next, we turn to the convergence analysis.  In Proposition~\ref{sdep}
we shall give error estimates for the difference
between the semidiscrete finite element solution and the elliptic projection
of the exact solution of the evolution equations.  Combining
these estimates with the estimates from Section~\ref{sec:ep} for the elliptic projection
gives  error estimates for the semidiscrete finite element method, which
we present in Theorem~\ref{semiconver}.

Assume the conditions of
Theorem \ref{sreg} hold, so the exact solution
$$
u\in C([0,T];L^2\Lambda^k)\cap C^1((0,T];D(L)).
$$
For
each $t>0$, we can then define the elliptic projection of $u(t)$ and of $u_t(t)$; see \eqref{ep1}--\eqref{ep3}.
Writing $(\hat\sigma_h(t),\hat u_h(t),\hat p_h(t))$ for the former, it is easy to see that its time-derivative,
$(\hat\sigma_{h,t},\hat{u}_{h,t},\hat{p}_{h,t})$, is the elliptic projection of $u_t$.  From Theorems~\ref{ellp}
and \ref{eprate}
we obtain error estimates, such as
\begin{multline}\label{sd11}
 \|u_t-\hat{u}_{h,t}\|
\leq C(E(u_t)+E(P_{\mathfrak{H}}u_t)+\eta [E(du_t)+E(\sigma_t)]
\\+(\eta^2+\beta)E(d\sigma_t)+\mu E(P_{\mathfrak{B}}u_t))
\leq C h^r\|u_t\|_{\bar H^r},
\end{multline}
with the last inequality holding for the choice of spaces made
in Theorem~\ref{eprate} (and similar results holding for the other allowable
choices of spaces).
Now, from \eqref{ep1},
\begin{equation}\label{sd8}
\<\hat{\sigma}_h,\tau\>-\<  d\tau,\hat{u}_h\>=0,\quad \tau\in \Lambda_h^{k-1}, \ t\in(0,T],
\end{equation}
and, substituting $Lu=-u_t+f$ into \eqref{ep2},
\begin{equation}\label{sd9}
\<\hat{u}_{h,t},v\>+\< d\hat{\sigma}_h,v\>+\< d\hat{u}_h,d v\>
=\<\hat{u}_{h,t}-u_t,v\>+\< f,v\>-\< \hat{p}_h,v\>.
\end{equation}
Define
$$
\sigerr=\hat{\sigma}_h-\sigma_h,\quad \uerr=\hat{u}_h-u_h,
$$
the difference between the elliptic projection and the finite element solution.
Subtracting \eqref{sd1} and \eqref{sd2} from \eqref{sd8} and \eqref{sd9}, respectively,
gives
\begin{gather}
 \< \sigerr,\tau\>-\< d\tau,\uerr\>=0,\quad \tau\in \Lambda_h^{k-1},\ 0<t\leq T,\label{sd12}
\\
\< \uerrt,v\>+\< d\sigerr,v\>+\< d \uerr,d v\>=\<\hat{u}_{h,t}-u_t-\hat{p}_h,v\>,
 \ v\in \Lambda_h^k,\ 0<t\leq T.\label{sd13}
\end{gather}
We shall now use these equations to derive bounds on $\Sigma_h$ and $U_h$ in terms of
$\hat{u}_{h,t}-u_t$ and $\hat{p}_h$, for which we derived bounds in Section~\ref{sec:ep}.
In the remainder of the paper, we adopt the notation $\|\,\cdot\|_{L^\infty(L^2)}$ for
the norm in $L^\infty(0,T;L^2\Lambda^k(\Omega))$
and similarly for other norms.
\begin{proposition}\label{sdep}
Assume $u_0\in D(L)$.  Then
\begin{gather*}
\|\uerr\|_{L^\infty(L^2)}+\|\sigerr\|_{L^2(L^2)}+
               \|d\uerr\|_{L^2(L^2)}
\leq C(\|\uerr(0)\|+\|\hat{u}_{h,t}-u_t-\hat{p}_h\|_{L^1(L^2)}),\\
 \|\sigerr\|_{L^\infty(L^2)}+\|d\sigerr\|_{L^2(L^2)}
\leq C(\|d_h^*\uerr(0)\|+\|\hat{u}_{h,t}-u_t-\hat{p}_h\|_{L^2(L^2)}).
\end{gather*}
\end{proposition}
\begin{proof}
By Theorem~\ref{sf}, $u\in C([0,T];D(L))\cap C^1([0,T];L^2\Lambda^k)$.
For each $t\in (0,T]$,  take $\tau=\sigerr(t)\in \Lambda_h^{k-1}$ in \eqref{sd12}
and $v=\uerr(t)\in \Lambda_h^k$ in \eqref{sd13},  and add to obtain
\begin{equation}
 \frac{1}{2}\frac{d}{dt}\|\uerr\|^2+\|\sigerr\|^2+\|d\uerr\|^2=\<\hat{u}_{h,t}-u_t-\hat{p}_h,\uerr\>,\label{sd10}
\end{equation}
which implies
\begin{equation*}
\frac{d}{dt}\|\uerr\|^2 \le 2\|\hat{u}_{h,t}-u_t-\hat{p}_h\|\|\uerr\|.
\end{equation*}
Taking $t^*\in [0,T]$ such  $\|\uerr\|_{L^\infty(L^2)}=\|\uerr(t^*)\|$,
and integrating this inequality from 0 to $t^*$ gives
\begin{equation*}
 \|\uerr(t^*)\|^2\leq \|\uerr(0)\|^2+2\|\hat{u}_{h,t}-u_t-\hat{p}_h\|_{L^1(L^2)}\|\uerr\|_{L^\infty(L^2)},
\end{equation*}
whence
\begin{equation}
\|\uerr\|_{L^\infty(L^2)} \leq \|\uerr(0)\|+ 2\|\hat{u}_{h,t}-u_t-\hat{p}_h\|_{L^1(L^2)},\label{sec4:22}
\end{equation}
which gives the desired bound on $U_h$.  To get the bound on $\Sigma_h$ and $dU_h$,
integrate \eqref{sd10} over $t\in [0,T]$.  This gives
\begin{equation*}
 \|\sigerr\|^2_{L^2(L^2)}+\|d\uerr\|^2_{L^2(L^2)}
\leq \frac{1}{2}\|\uerr(0)\|^2+ \|\uerr\|_{L^\infty(L^2)}\|\hat{u}_{h,t}-u_t-\hat{p}_h\|_{L^1(L^2)},
\end{equation*}
and so, by \eqref{sec4:22},
\begin{equation*}
 \|\sigerr\|_{L^2(L^2)}+\|d\uerr\|_{L^2(L^2)}\\
\leq C(\|\uerr(0)\|+\|\hat{u}_{h,t}-u_t-\hat{p}_h\|_{L^1(L^2)}),
\end{equation*}
which completes the proof of the first inequality.

To prove the second inequality, we
differentiate \eqref{sd12} in time and take $\tau=\sigerr\in \Lambda_h^{k-1}$, and then
add to \eqref{sd13} with $v=d\sigerr\in \Lambda_h^k$ (here we use the subcomplex property
$d\Lambda_h^{k-1}\subset \Lambda_h^k$).  This gives
\begin{equation*}
 \frac{1}{2}\frac{d}{dt}\|\sigerr\|^2+\|d\sigerr\|^2
=\<\hat{u}_{h,t}-u_t- \hat{p}_h,d\sigerr\>.
\end{equation*}
By integrating in time, first over $[0,t^*]$ with
$t^*\in [0,T]$ chosen so that $\|\sigerr\|_{L^\infty(L^2)}=\|\sigerr(t^*)\|$, and then over all of $[0,T]$,
we deduce that
\begin{equation*}
\|\sigerr\|_{L^\infty(L^2)} + \|d\sigerr\|_{L^2(L^2)}\leq
C(\|\sigerr(0)\|+\|\hat{u}_{h,t}-u_t-\hat{p}_h\|_{L^2(L^2)}).
\end{equation*}
Finally, we note from
\eqref{sd1} and \eqref{ep1} that $\sigerr=d_h^*\uerr$, and so complete the proof.
\end{proof}

Now suppose, for simplicity, that we choose the initial data $u_h^0$ to equal the elliptic
projection of $u_0$.  Then $U_h(0)=0$ and the right-hand sides of the inequalities in
Proposition~\ref{sdep} simplify.  Bounding them using Theorem~\ref{eprate}
and \eqref{sd11} we get, for the choice of spaces indicated in the theorem,
\begin{gather*}
\|\uerr\|_{L^\infty(L^2)}+\|\sigerr\|_{L^2(L^2)}+
               \|d\uerr\|_{L^2(L^2)}
\leq Ch^r(\|u_t\|_{L^1(\bar H^r)}+\|dd^*u\|_{L^1(H^{r-2})}),
\\
 \|\sigerr\|_{L^\infty(L^2)}+\|d\sigerr\|_{L^2(L^2)}
\leq Ch^r(\|u_t\|_{L^2(\bar H^r)}+\|dd^*u\|_{L^2(H^{r-2})}).
\end{gather*}

Combining these estimates with the estimates in Theorem~\ref{eprate} for
the elliptic projection, we obtain the main result of the section.
\begin{theorem}\label{semiconver}
Suppose that, in addition to the hypotheses of
Theorem~\ref{eprate} and \ref{sreg}, $u_0\in D(L)$.
Let $(\sigma,u)$ be the solution of (\ref{mwf1}), (\ref{mwf2}) satisfying (\ref{ic}),
and $(\sigma_h,u_h)$ the solution of
(\ref{sd1}), (\ref{sd2}) with the spaces selected as in Theorem~\ref{eprate}
and $u_h(0)$ chosen to be equal to the elliptic projection of $u_0$.
Then, we have the following error estimates for the semidiscrete finite element method:
\begin{align*}
\|\sigma-\sigma_h\|_{L^2(L^2)}&\leq Ch^r(\|u_t\|_{L^1(\bar H^r)} +\|d^*u\|_{L^2(H^r)}),
\\
\|\sigma-\sigma_h\|_{L^{\infty}(L^2)}&\leq Ch^r(\|u_t\|_{L^2(\bar H^r)}
 +\|d^*u\|_{L^\infty(H^r)}),
\\
\|d(\sigma-\sigma_h)\|_{L^2(L^2)} &\leq Ch^r(\|u_t\|_{L^2(\bar H^r)}
 +\|dd^*u\|_{L^2(H^r)}),
\\
 \|u-u_h\|_{L^{\infty}(L^2)}&\leq Ch^r(\|u\|_{L^\infty(\bar H^r)} + \|u_t\|_{L^1(\bar H^r)}),
\\
\|d(u-u_h)\|_{L^2(L^2)}&\leq Ch^r(\|u_t\|_{L^1(\bar H^r)} +\|du\|_{L^2(H^r)} + \|dd^*u\|_{L^2(H^{r-1})}).
\end{align*}
\end{theorem}

\section{The fully discrete finite element method}\label{sec:fd}
If we combine the semidiscrete finite element method with a standard time-stepping scheme to solve
the resulting system of ordinary differential equations, we obtain a
fully discrete finite element method for the Hodge heat equation
\eqref{mwf1}, \eqref{mwf2}. For simplicity, we use backward Euler's method
with constant time step $\Delta t=T/M$. We may choose any of the pairs
of finite element spaces indicated in \eqref{spaces} for any value of $r\ge 1$, but,
as above, for simplicity we restrict ourselves to
the choice $\Lambda^{k-1}_h=\P_r^-\Lambda^{k-1}(\Th)$ and
$\Lambda^k_h=\P_r^-\Lambda^k(\Th)$ with $r>1$,
the results for the other cases being simple variants.  The fully discrete method seeks
$\sigma_h^n\in \Lambda_h^{k-1},u_h^n\in \Lambda_h^k$,
satisfying the equations
\begin{gather}
 \<\sigma_h^n,\tau\>-\< d\tau,u_h^n\>=0,\quad \tau\in \Lambda_h^{k-1},\label{sec5:eq1}\\
\<\frac{u_h^n-u_h^{n-1}}{\Delta t},v\>+\< d\sigma^n_h, v\>+\< du^n_h, dv\>=\< f(t^n),v\>,
\quad v\in \Lambda_h^k.\label{sec5:eq2}
\end{gather}
for $1\le n\le M$.
It is easy to see that this linear system for $u_h^n, \sigma_h^n$ is invertible at each time step.
We initialize by choosing $u_h^0\in \Lambda_h^k$.  We also define $\sigma_h^0\in\Lambda^{k-1}$
so that \eqref{sec5:eq1} holds for $n=0$.

Next, we turn to the convergence analysis.  We first obtain
error estimates for the difference between the fully discrete finite element solution and the elliptic
projection of the exact solution of the evolution equations.  These are
stated in \eqref{fulrate1} and \eqref{fulrate2}.  Combining
these estimates with the estimates from Section~\ref{sec:ep} for the elliptic projection,
we obtain the error estimates for the fully discrete finite element method
presented in Theorem~\ref{fullyconver}.

The analysis is similar to that for the semidiscrete finite element method, but with some
extra complications arising from the time discretization. Let
($\hat{\sigma}_h^n, \hat{u}^n_h,\hat{p}_h^n$) be the elliptic projection of $u^n=u(t^n)$.

Now, from \eqref{ep1}
\begin{equation}
\<\hat{\sigma}^n_h,\tau\>-\< d\tau,\hat{u}^n_h\>=0,
  \quad \tau\in \Lambda_h^{k-1}, \  0\leq n\leq M,\label{sec5:eq3}
\end{equation}
and,
from \eqref{ep2} and the equation $u_t + Lu =f$,
\begin{multline}
\<\frac{\hat{u}_h^n-\hat{u}_h^{n-1}}{\Delta t},v\>+\< d\hat{\sigma}^n_h,v\>+\< d\hat{u}^n_h,d v\>
=\<\frac{\hat{u}_h^n-\hat{u}_h^{n-1}}{\Delta t}-u_t^n,v\>+\< f^n,v\>-\< \hat{p}^n_h,v\>
\\
=\<\frac{(\hat{u}_h^n-u^n)-(\hat{u}_h^{n-1}-u^{n-1})}{\Delta t},v\>
+\<\frac{u^n-u^{n-1}}{\Delta t}-u_t^n,v\>
+\< f^n,v\>-\< \hat{p}^n_h,v\>, \\ v\in \Lambda_h^k,\  1\leq n\leq M.\label{sec5:eq4}
\end{multline}
Set
$$
\sigerr^n=\hat{\sigma}^n_h-\sigma^n_h,\quad  \uerr^n=\hat{u}^n_h-u^n_h,
$$
the difference
between the elliptic projection and the finite element solution at each time step.
Subtracting \eqref{sec5:eq1} and \eqref{sec5:eq2} from \eqref{sec5:eq3} and
\eqref{sec5:eq4}, respectively, gives
\begin{equation}
 \< \sigerr^n,\tau\>-\< d\tau,U^n_h\>=0,\quad \tau\in \Lambda_h^{k-1},\ 0\leq n\leq M, \label{sec5:eq5}
\end{equation}
and
\begin{equation}\label{sec5:eq6}
\begin{aligned}
\<\frac{\uerr^n-\uerr^{n-1}}{\Delta t}&,v\>+\< d\sigerr^n,v\>+\< d\uerr^n,d v\>
\\
&=\<\frac{(\hat{u}_h^n-u^n)-(\hat{u}_h^{n-1}-u^{n-1})}{\Delta t},v\>+\<\frac{u^n-u^{n-1}}{\Delta t}-u_t^n,v\>
-\< \hat{p}^n_h,v\> \\
&=\< z^n,v\>, \quad v\in \Lambda_h^k,\ 1\leq n\leq M,
\end{aligned}
\end{equation}
where $z^n\in L^2\Lambda^k$ is defined by the last equation.
We easily see that
\begin{equation*}
\|z^n\|
\leq \frac{1}{\Delta t}\int_{t^{n-1}}^{t^n}\|(\hat{u}_{h,t}-u_t)(s)\|ds
+\frac{\Delta t}{2}\|u_{tt}\|_{L^{\infty}(L^2)}+\|\hat{p}^n_h\|.
\end{equation*}
By Theorem~\ref{eprate}, the last term on the right hand side is bounded
by $Ch^r\|dd^*u\|_{L^{\infty}(H^{r-2})}$, and, by \eqref{sd11}, the first term on the right hand
side by
$$
\frac{Ch^r}{\Delta t}\|u_t\|_{L^1([t^{n-1},t^n],\bar H^r)}.
$$
Thus we have proved:
\begin{proposition}\label{Eerr}
\begin{equation*}
\|z^n\|\leq \frac{\Delta t}{2}\|u_{tt}\|_{L^{\infty}(L^2)}+Ch^r(\|dd^*u\|_{L^{\infty}(H^{r-2})}
  +\frac{1}{\Delta t}\|u_t\|_{L^1([t^{n-1},t^n],\bar H^r)}).
\end{equation*}
\end{proposition}

We shall now use equations \eqref{sec5:eq5} and \eqref{sec5:eq6} to derive bounds on
$\sigerr$ and $\uerr$ in terms of $z$.
Toward this end we adopt the notations
$$
\|f\|_{l^\infty(X)}=\max_{1\le n\le M}\|f^n\|_X,\
\|f\|_{l^2(X)}=(\Delta t\sum_{n=1}^M \|f^n\|_X^2)^{1/2}, \
\|f\|_{l^1(X)} = \Delta t \sum_{n=1}^M \|f^n\|_X.
$$
\begin{proposition}\label{fdep}
 Assume $u_0\in D(L)$. Then
 \begin{gather*}
 \|\uerr\|_{l^\infty(L^2)}+ \|\sigerr\|_{l^2(L^2)}+\|d\uerr\|_{l^2(L^2)}
\leq C(\|U_h(0)\|+\|z\|_{l^1(L^2)}),
\\ \|\sigerr\|_{l^\infty(L^2)}+\|d\sigerr\|_{l^2(L^2)}\leq C(\|d_h^*U_h(0)\|+\|z\|_{l^2(L^2)}).
\end{gather*}
\end{proposition}

\begin{proof}
Take $\tau=\sigerr^n\in \Lambda_h^{k-1}$ in \eqref{sec5:eq5}, $v=\uerr^n\in \Lambda_h^k$ in \eqref{sec5:eq6},
and add to obtain
\begin{equation}
 \Delta t(\|\sigerr^n\|^2+\|d\uerr^n\|^2)+\|\uerr^n\|^2=(\uerr^{n-1}+\Delta tz^n,\uerr^n),\label{sec5:eq8}
\end{equation}
which implies
\begin{equation*}
 \|\uerr^n\|\leq \|\uerr^{n-1}\|+\Delta t \|z^n\|.
\end{equation*}
By iteration,
\begin{equation}
\|\uerr\|_{l^\infty(L^2)}\leq \|\uerr^0\|+\Delta t\sum_{n=1}^M\|z^n\|,
\label{sec5:eq7}\end{equation}
which is the desired bound on $U_h$. To get the bound on $\sigerr$ and $d\uerr$,
we derive from \eqref{sec5:eq8} that
\begin{equation*}
      \frac{1}{2}\|\uerr^n\|^2 - \frac{1}{2}\|\uerr^{n-1}\|^2
   +  \Delta t(\|\sigerr^n\|^2+\|d\uerr^n\|^2)\leq (\Delta tz^n,\uerr^n)
      \le \Delta t \|z^n\| \|U_h\|_{l^\infty(L^2)} .
\end{equation*}
Summing then gives
\begin{equation*}
      \frac{1}{2}\|\uerr^M\|^2 - \frac{1}{2}\|\uerr^0\|^2 + \|\sigerr\|_{l^2(L^2)}^2+\|d\uerr\|_{l^2(L^2)}^2\leq
 \|\uerr\|_{l^\infty(L^2)}\|\|z\|_{l^1(L^2)},
\end{equation*}
and so,  by  \eqref{sec5:eq7},
\begin{equation*}
\|\sigerr\|_{l^2(L^2)}^2+\|d\uerr\|_{l^2(L^2)}^2
\leq \|\uerr^0\|^2+\frac{3}{2}\|z\|_{l^1(L^2)}^2,
\end{equation*}
which completes the proof of the first inequality.

To prove the second inequality, we take $\tau=\sigerr^n\in \Lambda_h^{k-1}$ in \eqref{sec5:eq5} at both
time level $n-1$ and level $n$. This gives
\begin{equation}
(\sigerr^{n-1},\sigerr^n)=(d\sigerr^n,\uerr^{n-1}),\quad
(\sigerr^{n},\sigerr^n)=(d\sigerr^n,\uerr^{n}),\quad 1\leq n\leq M.\label{sec5:fdsig}
\end{equation}
Next take $v=d\sigerr^n\in \Lambda_h^k$ in \eqref{sec5:eq6} and substitute \eqref{sec5:fdsig} to get
\begin{equation*}
 (\sigerr^{n}-\sigerr^{n-1},\sigerr^n)+\Delta t \|d\sigerr^n\|^2=\Delta t(z^n,d\sigerr^n),\ 1\leq n\leq M,
\end{equation*}
whence
\begin{equation*}
 \|\sigerr^{n}\|^2-\|\sigerr^{n-1}\|^2+\Delta t \|d\sigerr^n\|^2\leq \Delta t\|z^n\|^2.
\end{equation*}
Again we get a telescoping sum, so
\begin{equation*}
\|\sigerr\|_{l^\infty(L^2)}^2+\|d\sigerr\|_{l^2(L^2)}^2
\leq C(\|\sigerr^0\|^2+\|z\|_{l^2(L^2)}^2).
\end{equation*}
This implies the second inequality and so completes the proof of the proposition.
\end{proof}

As in Section~\ref{sec:sd}, we choose the initial data $u_h^0$ to equal the elliptic projection of $u_0$ for
simplicity. Then $U_h(0)=0$ and the right-hand sides of the inequalities in Proposition~\ref{fdep} simplify. Bounding them
via Proposition~\ref{Eerr} we get for the first
\begin{multline}\label{fulrate1}
 \|\uerr\|_{l^\infty(L^2)}+\|\sigerr\|_{l^2(L^2)}+\|d\uerr\|_{l^2(L^2)}
\\
\leq C\Delta t\|u_{tt}\|_{L^{\infty}(L^2)}+Ch^r(\|dd^*u\|_{L^{\infty}(H^{r-2})}+\|u_t\|_{L^1(\bar H^r)}).
\end{multline}
For the second, we bound the $L^1([t^{n-1},t^n])$ norm in Proposition~\ref{Eerr} by $\Delta t$
times the $L^\infty$ norm, and substitute the resulting bound for $z$ in the second estimate of
Proposition~\ref{fdep}, obtaining
\begin{multline}\label{fulrate2}
\|\sigerr\|_{l^\infty(L^2)}+\|d\sigerr\|_{l^2(L^2)}
\leq
C\Delta t\|u_{tt}\|_{L^{\infty}(L^2)}+Ch^r(\|dd^*u\|_{L^{\infty}(H^{r-2})}
+\|u_t\|_{L^{\infty}(\bar H^r)}).
\end{multline}

Combining \eqref{fulrate1}, \eqref{fulrate2} with the estimates in Theorem~\ref{eprate}
for the elliptic projection, we obtain the main result of the section.
\begin{theorem}\label{fullyconver}
Under the same assumptions as Theorem~\ref{semiconver},
Let $(\sigma,u)$ be the solution of (\ref{mwf1}), (\ref{mwf2}) satisfying (\ref{ic}),
and $(\sigma_h^n,u_h^n)$ the solution of
(\ref{sec5:eq1}), (\ref{sec5:eq2}) with $u_h^0$ equal to the elliptic projection of $u_0$.
Then, we have the following error estimates for the fully discrete finite element method:
\begin{gather*}
\|\sigma-\sigma_h\|_{l^2(L^2)} \leq
 C\Delta t\|u_{tt}\|_{L^{\infty}(L^2)}
+Ch^r(\|u_t\|_{L^1(\bar H^r)}+\|d^*u\|_{ L^{\infty}(H^r)}),
\\
\|\sigma-\sigma_h\|_{l^\infty(L^2)} \leq
 C\Delta t\|u_{tt}\|_{L^{\infty}(L^2)}
+Ch^r(\|u_t\|_{L^{\infty}(\bar H^r)}+\|d^*u\|_{L^{\infty}(H^r)} ),
\\
\|d(\sigma-\sigma_h)\|_{l^2(L^2)} \leq
 C\Delta t\|u_{tt}\|_{L^{\infty}(L^2)}+Ch^r(\|u_t\|_{L^{\infty}(\bar H^r)}+\|dd^*u\|_{L^{\infty}(H^r)}),\\
\|u-u_h\|_{l^\infty(L^2)} \leq
 C\Delta t\|u_{tt}\|_{L^{\infty}(L^2)}
+Ch^r(\|u\|_{L^{\infty}(\bar H^r)}+\|u_t\|_{L^1(\bar H^r)})),
\\
 \|d(u-u_h)\|_{l^2(L^2)} \leq     C\Delta t\|u_{tt}\|_{L^{\infty}(L^2)}
+Ch^{r}(\|u_t\|_{L^1(\bar H^{r})}+\|du\|_{ L^{\infty}(H^{r})}+\|dd^*u\|_{L^{\infty}(H^{r-1})}).
\end{gather*}
\end{theorem}

The error estimates are analogous to those of Theorem~\ref{semiconver} for the
semidiscrete solution, with each containing an additional $O(\Delta t)$ term coming from
the time discretization.  For each quantity, the error is of order $O(\Delta t +h^r)$.

\section{Numerical examples}\label{sec:num}
In this section, we present the results of simple numerical computations
verifying the theory above.

First we compute a two-dimensional example for the $1$-form Hodge heat equation.
Using vector proxies, we may write the parabolic equations \eqref{hhe}--\eqref{ic} as
\begin{gather*}
u_t+(\curl\rot-\nabla\div)u=f \text{ in } \Omega\times [0,T],\\
u\cdot n=\rot u=0 \text{ on } \partial \Omega\times [0,T], \quad
u(\,\cdot\,,0)=u_0  \text{ in } \Omega,
\end{gather*}
where
$$
\rot u=\frac{\partial u_2}{\partial x_1}-\frac{\partial u_1}{\partial x_2}, \quad
\curl u=(\frac{\partial u}{\partial x_2},-\frac{\partial u}{\partial x_1}).
$$
We choose $\Omega$ to be a square annulus $[0,1]\times[0,1]\backslash[0.25,0.75]\times[0.25,0.75]$ and take
 the exact solution as
$$
u=\begin{pmatrix}
100x(x-1)(x-0.25)(x-0.75)t\\
100 y(y-1)(y-0.25)(y-0.75)t
        \end{pmatrix}.
$$
Note that this function is not orthogonal to 1-harmonic forms on $\Omega$.
We use the finite element spaces $\P_r\Lambda^0(\Th)$ (Lagrange elements of degree $r$) for
$\sigma=-\div u$ and $\P_r^-\Lambda^1(\Th)$ (Raviart--Thomas elements) for $u$, starting
with an initial unstructured  mesh, and then refining it uniformly.
We take  $\Delta t=0.0001$ and compute the error at time $T=0.01$ (after 100 time steps).
Tables~\ref{tb:1} and \ref{tb:2} show the results for $r=1$ and $2$ respectively.
The rates of convergence are just as predicted by the theory.
\begin{table}[htbp]
\begin{tabular}{ccccccc}
mesh size&  $\|\sigma-\sigma_h\|$&  rate& $\|\nabla(\sigma-\sigma_h)\|$&rate&$\|u-u_h\|$&rate\\
\hline\\[-2ex]
$h$&   0.0008490  &1.99&0.1026276&1.01&0.0010586&0.96\\
$h/2$&      0.0002132&  1.99&0.0512846&1.00&0.0005341 &0.99\\
$h/4$&0.0000534&2.00&0.0256528 &1.00&0.0002678 &1.00\\
$h/8$&0.0000133 &2.00&0.0128295 &1.00&0.0001340 &1.00\\
\hline\\
\end{tabular}
\caption{Computation with $P_1\Lambda^0\times P_1^-\Lambda^1$ in two dimensions.}\label{tb:1}
\end{table}

\begin{table}[htbp]
\begin{tabular}{ccccccc}
mesh size&  $\|\sigma-\sigma_h\|$&  rate& $\|\nabla(\sigma-\sigma_h)\|$&rate&$\|u-u_h\|$&rate\\
\hline\\[-2ex]
$h$&  0.0000093   &3.03&0.0016510&2.03&0.0000705&1.97\\
$h/2$&      0.0000012&  3.00&0.0004119&2.00&0.0000178 &1.99\\
$h/4$&0.0000001&3.00&0.0001031 &2.00&0.0000045 &1.99\\
$h/8$&0.0000000  &3.04&0.0000258 &2.00&0.0000011 &2.00\\
\hline\\
\end{tabular}
\caption{Computation with $P_2\Lambda^0\times P_2^-\Lambda^1$ in two dimensions.}\label{tb:2}
\end{table}

For the second example, we let
$\Omega$ be the unit cube $[0,1]\times[0,1]\times [0,1]$ in $\mathbb{R}^3$, and
again solve the $1$-form Hodge heat equation.
Using vector proxies, the initial--boundary value problem becomes
\begin{gather*}u_t+(\curl\curl-\nabla\div)u=f \text{ in } \Omega\times [0,T]\\
u\cdot n=0,\ \curl u \x n=0 \text{ on } \partial \Omega\times [0,T], \quad
u(\,\cdot\,,0)=u_0  \text{ in } \Omega.
\end{gather*}
We take the exact solution to be
$$
u=\begin{pmatrix}
\sin (\pi x_1)t\\
\sin (\pi x_2)t\\
\sin (\pi x_3)t
\end{pmatrix}.
$$
Table~\ref{tb:3} shows the errors and rates of convergence for linear elements
on a sequence of uniform meshes, again at time $T=0.01$ after $100$ time steps.
Once again, the rates of convergence
are just as predicted by the theory.

\begin{table}[htbp]
\begin{tabular}{ccccccc}
mesh size&  $\|\sigma-\sigma_h\|$&  rate& $\|\nabla(\sigma-\sigma_h)\|$&rate&$\|u-u_h\|$&rate\\
\hline\\[-2ex]
\phantom{00}0.25& 0.0023326  &2.06&0.0260155&1.02&0.0026024&1.00\\
\phantom{0}0.125&     0.0005735&  2.02&0.0134836&0.95&0.0013499 &0.95\\
0.0625&0.0001429 &2.01&0.0068169  &0.98&0.0006879&0.97\\
\hline\\
\end{tabular}
\caption{Computation with $P_1\Lambda^0\times P_1^-\Lambda^1$ in three dimensions.}\label{tb:3}
\end{table}
\providecommand{\bysame}{\leavevmode\hbox to3em{\hrulefill}\thinspace}
\providecommand{\MR}{\relax\ifhmode\unskip\space\fi MR }
\providecommand{\MRhref}[2]{%
  \href{http://www.ams.org/mathscinet-getitem?mr=#1}{#2}
}
\providecommand{\href}[2]{#2}


\begin{thebibliography}{10}

\bibitem{MR2269741}
Douglas~N. Arnold, Richard~S. Falk, and Ragnar Winther, \emph{Finite element
  exterior calculus, homological techniques, and applications}, Acta Numer.
  \textbf{15} (2006), 1--155. \MR{2269741 (2007j:58002)}

\bibitem{MR2594630}
\bysame, \emph{Finite element exterior calculus: from {H}odge theory to
  numerical stability}, Bull. Amer. Math. Soc. (N.S.) \textbf{47} (2010),
  no.~2, 281--354. \MR{2594630 (2011f:58005)}

\bibitem{MR2759829}
Haim Brezis, \emph{Functional analysis, {S}obolev spaces and partial
  differential equations}, Universitext, Springer, New York, 2011. \MR{2759829}

\bibitem{MR1691574}
Thierry Cazenave and Alain Haraux, \emph{An introduction to semilinear
  evolution equations}, Oxford Lecture Series in Mathematics and its
  Applications, vol.~13, The Clarendon Press Oxford University Press, New York,
  1998. \MR{1691574 (2000e:35003)}

\bibitem{MR2373181}
Snorre~H. Christiansen and Ragnar Winther, \emph{Smoothed projections in finite
  element exterior calculus}, Math. Comp. \textbf{77} (2008), no.~262,
  813--829. \MR{2373181 (2009a:65310)}

\bibitem{MR0277126}
Jim Douglas, Jr. and Todd Dupont, \emph{Galerkin methods for parabolic
  equations}, SIAM J. Numer. Anal. \textbf{7} (1970), 575--626. \MR{0277126 (43
  \#2863)}

\bibitem{GL}
L.~Gor'kov and G.~\'Eliashberg, \emph{Generalization of the
  {G}inzburg--{L}andau equations for non-stationary problems in the case of
  alloys with paramagnetic impurities}, Soviet Phys.-JETP \textbf{27} (1968),
  328--334.

\bibitem{holst-gillette}
Michael Holst and Andrew Gillette, \emph{Finite element exterior calculus for
  evolution equations}, arXiv preprint 1202.1573, 2012.

\bibitem{MR610597}
Claes Johnson and Vidar Thom{\'e}e, \emph{Error estimates for some mixed finite
  element methods for parabolic type problems}, RAIRO Anal. Num\'er.
  \textbf{15} (1981), no.~1, 41--78. \MR{610597 (83c:65239)}

\bibitem{MR0351124}
Mary~Fanett Wheeler, \emph{A priori {$L_{2}$} error estimates for {G}alerkin
  approximations to parabolic partial differential equations}, SIAM J. Numer.
  Anal. \textbf{10} (1973), 723--759. \MR{0351124 (50 \#3613)}

\bibitem{MR2784829}
Scott~O. Wilson, \emph{Differential forms, fluids, and finite models}, Proc.
  Amer. Math. Soc. \textbf{139} (2011), no.~7, 2597--2604. \MR{2784829}

\end{thebibliography}
 \end{document}